\documentclass[12pt]{article}
\usepackage{amssymb,amsmath,amsthm,graphicx, color, amsfonts}
\usepackage[english]{babel}
\def\supp{\mathop{\rm supp}\nolimits}
\newtheorem{theorem}{Theorem}[section]

\newtheorem{proposition}[theorem]{Proposition}
\newtheorem{corollary}[theorem]{Corollary}
\newtheorem{definition}[theorem]{Definition}
\newtheorem{remark}[theorem]{Remark}
\newtheorem{example}[theorem]{Example}

\renewenvironment{proof}[1][.]{%
\bigskip\noindent{\bf Proof#1 }}{%
\hfill$\blacksquare$\bigskip}
\begin{document}
\title{Continuous groupoids on the symbolic space, quasi-invariant probabilities for Haar systems and the Haar-Ruelle operator}



\author{Artur O. Lopes and Elismar R. Oliveira}
\date{}
\maketitle

\noindent\rule{\textwidth}{0.1mm}
\begin{abstract}
We consider groupoids on $\{1,2,..,d\}^\mathbb{N}$, cocycles and  the counting measure as transverse function.
We  generalize  results  relating  quasi-invariant probabilities with eigenprobabilities for the dual of the Ruelle operator. We assume a mild compatibility of the groupoid with the symbolic structure.
We present a   generalization of the Ruelle operator - the Haar-Ruelle operator - taking into account the Haar structure. We consider continuous and also  H\"older cocycles.  IFS with weights appears in our reasoning in the H\"older case.

\end{abstract}
\vspace {0.1mm}
\noindent\rule{\textwidth}{0.1mm}

\tableofcontents

\vspace {0.1mm}

\section{Introduction}

The symbolic space $X=\{1,2,...,d\}^{\mathbb{N}}$ is a compact metric space for the usual metric $d$ described on \cite{PP}.  The  shift $\sigma$ acts on this space as a continuous transformation.

We introduce an equivalence relation $\sim$ in $X$, such that, the map
$x \to [x]:=\{s \in X | s \sim x\}$
is hemicontinuous, as a multivalued function (see \cite{AB}, Section 17.2, for details on hemicontinuity and nets).

\begin{definition} Given an equivalence relation $\sim$ we consider the associated groupoid
$$G=\{\,\,(x,y) \in X^2 \,\,\text{such that} \,\, x \sim y\,\,\},$$
and, as usual we denote $G^0=\{\,\,(x,x) \in G \,\,\text{such that}\,\,  x \in X \,\,\,\}\cong X$.
\end{definition}

We identify  $X$ with $G^0$.

We consider over $G$ the topology induced by the  product topology on $X^2$.
$\mathcal{B}$ denotes the Borel sigma-algebra on $G$.

In general groupoids are equipped with some algebraic structure (see \cite{Con} or \cite{Ren2}) but this will not be relevant for our purposes (see discussion on \cite{CELM}).

\begin{definition} A measurable (continuous) groupoid $G$ is a groupoid, such that,

$P_1(x,y)=x$, $\,P_2(x,y)=y$,  $\,h(x,y)=(y,x)$ and $Z (\,(x,s),(s,y)\,)= (x,y),$
are  Borel measurable (continuous).

\end{definition}

\begin{definition}
A transverse function  $\nu$ on the measurable groupoid $G$ is a map  of $X$ in the space of measures over the sigma-algebra $\mathcal{B}$, such
that,

1)  $\forall y \in X$, the measure $\nu^y$ has support on $[y]$,

2) $\forall A \in \mathcal{B}$, we have that $\nu^y (A)$, as a function of $y$, is measurable.

3) for any $r,s\in[y]$ we have that $\nu^s=\nu^r.$

\end{definition}

General references on groupoids and transverse functions for an Ergodic Theory audience are
\cite{LM1} and \cite{CELM}.

The only transverse function $\nu$ we will consider here is the counting measure.

\begin{definition} Given the   transverse function $\nu$ described above over the measurable groupoid $G$ we consider the {\bf Haar system} $(G, \nu)$, where $\supp \nu^{x} =[x]$ and  $\nu^{x}=\nu^{y}$, when $x\sim y$.  We assume also that $\nu^{x}([x]) \leq \infty$.
\end{definition}

A modular function $\delta:G \to \mathbb{R}$ is a continuous function, such that, $\delta(x,  z)=\delta(x,  y)\, \delta(y,  z), \forall x\sim y\sim z \in X$ (see section 3 in \cite{Con}).

\begin{definition} $c: G \to \mathbb{R}$ is a cocycle, if $c(x,  z)=c(x,  y)+c(y,  z), \forall x\sim y\sim z \in X$. Therefore, $\delta(x,  y)=e^{\beta c(y,x)} = e^{-\beta c(x,y)}$ is a modular function, for any $\beta \in \mathbb{R}$.

We define $\mathcal{C}_{\sim}=\{ c: G \to \mathbb{R}\,\,| \,\,\text{ c is a cocycle for } \sim\}$.
\end{definition}

We claim that $\mathcal{C}_{\sim}$ is a linear space. Indeed, given $c,b \in \mathcal{C}_{\sim}$ and $\alpha \in \mathbb{R}$ we have
$$ (c + \alpha b) (x,z)=c(x,z) + \alpha b(x,z) = c(x,  y)+c(y,  z) + \alpha (b(x,  y)+b(y,  z))=$$
$$=c(x,  y)+ \alpha b(x,  y) + c(y,  z) + \alpha b(y,  z)= (c + \alpha b) (x,y) + (c + \alpha b) (y,z),$$
and, $0 \in \mathcal{C}_{\sim}$.

\begin{definition}We call coboundary  a function of the form $V(y)-V(x)\in \mathcal{C}_{\sim}$, for some $V: G^0 \to \mathbb{R}$. Given $c_0 \in \mathcal{C}_{\sim}$, we say that $c$ is cohomologous to $c_0$, if
$c(x,y)=c_0(x,y) + V(y)-V(x)$.
\end{definition}

\begin{definition}
We say that a cocycle $c$ is separable if it is cohomologous to $0$, that is,
$$c(x,y)=0 + V(y)-V(x)=V(y)-V(x).$$
\end{definition}

Motivated by examples in Statistical Mechanics and
Quantum Field Theory, the authors Kubo, Martin and Schwinger  introduce the concept of KMS state
on a $C^*$-Algebra or on a  Von Neumann Algebra. They describe on Quantum Statistical Physics the role of the Gibbs state (see \cite{CELM} or \cite{Ren2}).

A large class of Von Neumann Algebras are defined from measurable groupoids and  Haar systems. A cocycle - in some sense - plays in this setting the role of the external potential in Statistical Mechanics.

The quasi-invariant condition for a probability  $M$ on $X$ (to be defined next) is related to the so called KMS-condition and to KMS states on von Neumann algebras or $C^*$ algebras (see \cite{CELM}, \cite{EL}, \cite{EL1}  or \cite{Ren2}).

We will be interested here in quasi-invariant probabilities for a certain family of groupoids and a certain kind of  cocycle (see Definition 2.3.8. in \cite{Ren2}). It will be not necessary to talk about KMS states.
\begin{definition}
  Given a cocycle $c$ we say that probability $M$ over the Borel sets  of $X$ satisfies the $(c,\ \beta)-${\bf quasi-invariant condition} for the grupoid $(G,\ \nu)$, if for any integrable function $h:G\rightarrow \mathbb{R}$, we have
  \begin{equation} \label{forti} \int\int h(s,\ x)d\nu^{x}(s)dM(x)=\int\int h(x,\ s)e^{-\beta c(x,s)}d\nu^{x}(s)dM(x),\end{equation}
where  $\beta\in \mathbb{R}$ and $c:G\rightarrow \mathbb{R}.$
\end{definition}

A natural question is: given a cocycle $c$ is there a relation of the associated quasi-stationary probability $M$ (which is defined by (\ref{forti})) with the Gibbs state (for some potential) of Statistical Mechanics? The answer is yes and this is the primary interest here.
Our main results are
Theorems \ref{quasi inv non sep}, \ref{quasi inv sep},  \ref{reduct separable to IFS} and Proposition \ref{pro}.

One of our purposes here is to show the relation of quasi-invariant probabilities with eigenprobabilities of the dual of   a general form of Ruelle operator (which takes into account the transverse function $\nu$). A particular case of this kind of result appears on section 4 in \cite{CELM} which deals with the Classical Thermodynamic Formalism (see \cite{PP}). References for different forms of Ruelle operators (some of them for IFSw) are \cite{BCLMS}, \cite{CO}, \cite{FL}, \cite{LMMS} and \cite{MO}.

Among other things we will consider here the Haar-Ruelle operator (see Definition \ref{haar-ruelle op defin}) which is a natural concept to consider in the present setting.

We point out that the results about quasi-stationary probabilities of \cite{KumRen}
and \cite{Ren2} (on the setting of $C^*$-Algebras) have a different nature of the ones we consider here (there, for instance, for just one value of $\beta$ you get a quasi-invariant probability - Theorem 3.5 in \cite{KumRen}).  Moreover, in \cite{KumRen} for such value of $\beta$  the  KMS  (quasi-invariant) probability for the $C^*$ algebra is  unique.  Here the results are for any $\beta$ and quasi-invariant probabilities are not unique.

We will consider here a more general class of groupoids than \cite{CELM}. We will present in Example \ref {examp 3} the expression of the quasi-invariant probability for a certain cocycle using an iteration method which follow from our reasoning.
This is particularly important for results related to spectral triples (see \cite{CL1}) 

\medskip

In order to obtain a connection between the groupoid (the Haar system) and the symbolic structure on $X$ we will require a mild compatibility hypothesis on the equivalence relation $\sim$.

We recall that $X=\{1,2,...,d\}^{\mathbb{N}}$ and the operation $i*x=(i,x_1, x_2, ...)$ is the concatenation of the symbol $i$ in the first position displacing all the symbols in $x=(x_1, x_2, ...)$.

\begin{definition}
We say that the equivalence relation $\sim$ is {\bf continuous} with respect to the symbolic structure if,
for all $x \in X$,  the set given by
$$\{1*x, ..., d*x\} \in X^d$$
has a continuous  representation $\psi$ of its classes,  $j*x \stackrel{\psi}{\to} [j*x]$, in the following sense:\\

First note that $\psi( j*x)$ is an $k$-uple (where $k$ is the cardinality of $[j*x]$). We assume it is well defined an ordered string $\{\psi_1( j*x),...,\psi_k( j*x)\}$.

\medskip

$\forall \varepsilon>0$, there exists $\delta>0$, such that, if $d(x,z)<\delta$ and $$[j*x]=\{\psi_1(j*x), ..., \psi_k(j*x)\}$$ $$ [j*z]=\{\psi_1(j*z), ..., \psi_k(j*z)\},$$
then, $\displaystyle\max_{i=1..k} d(\psi_i(j*x),\psi_i(j*z)) < \varepsilon$, where $k=\sharp [j*x]= \sharp [j*z]$.

For a Lipschitz  relation $\sim$ will be required that for $\psi$:
$$\max_{i=1, ..., k} d(\psi_i(j*x),\psi_i(j*z)) < r d(x,z),$$
in particular, $Lip(\psi_i(j*x)) \leq r$, for all $i,j$.
\end{definition}

\begin{definition} A groupoid $G$ associated to a continuous equivalence $\sim$  relation will be called a {\bf continuous groupoid}. A Lipschitz groupoid is defined on a similar manner.

\end{definition}

We assume from now on that $G$ is at least a continuous groupoid.

\begin{proposition} \label{order relat struct}
  The representation $\psi$ has the  absorbtion property
  $$\psi_{a}(\psi_{b}(x)) = \psi_{a}(x),$$
  for all $x \in X$.

\end{proposition}
\begin{proof}
  The proof is obvious from the definition because $[\psi_{b}(x)]=[x]$ and $\psi_{a}$ returns the $a$-nth element in this class which is $\psi_{a}(x)$.
\end{proof}

\begin{example}
The equivalence relation, $x\sim y \Leftrightarrow \sigma x= \sigma y$, is  continuous with respect to the dynamic $\sigma$ because the correspondence
$$ j*x \stackrel{\psi}{\to}  [j*x]=\{1*x, ..., d*x\}$$
is constant, given by,
$\psi_i(j*x):=i*x,$
for all $j$.
\end{example}

\section{Non-separable cocycles and the Haar-Ruelle operator}

\begin{definition}\label{haar-ruelle op defin}
  Let $\varphi(x,y)$ be a  continuous function in $G$. We introduce the Haar-Ruelle operator, $L_{\varphi}$ as the operator
  $$L_{\varphi}(f)(x)= \frac{1}{d}\sum_{j=1}^{d} \int_{[j*x]} f(s)\, e^{\varphi(j*x,s)} d\nu^{j*x}(s),$$
  acting onfunctions $f:X \to \mathbb{R}$,  integrable with respect to the transverse function $\nu$.
\end{definition}

Example \ref{gen} will show the evidence that we are considering above a generalization of the classical Ruelle operator.

Let $c(x,y)$ be a general (continuous, Lipschitz, H\"older) cocycle (that is, we do not require that $c(x,y)= V(y)-V(x)$). We introduce the Haar-Ruelle operator, $L_{-\beta c}$ by choosing $\varphi=-\beta c$ in Definition~\ref{haar-ruelle op defin}, that is,
$$L_{-\beta c}(f)(x)= \frac{1}{d}\sum_{j=1}^{d} \int_{[j*x]} f(s)\, e^{-\beta c(j*x,s)} d\nu^{j*x}(s),$$
for any integrable $f:X \to \mathbb{R}$.

We recall that $[j*x]=\{\psi_1(j*x), ..., \psi_k(j*x)\}$ and $\nu^{j*x}$ is the counting measure so the  Haar-Ruelle operator takes the form:
$$L_{-\beta c}(f)(x)= \frac{1}{d}\sum_{j=1}^{d} \sum_{i=1}^{k} f(\psi_i(j*x))\, e^{-\beta c(j*x,\ \psi_i(j*x))}.$$

{\bf All of our results are true for any $\beta>0$.}

As usual the dual $L^*$ of an operator $L$ acting on continuous function acts on measures (see \cite{PP}).

\begin{theorem}\label{haar-ruelle propert}
  Consider the Haar-Ruelle operator, $L_{-\beta c}$. Then,\\
  a) $L_{-\beta c}$ is positive and preserves $C^0$,\\
  b) There exists $\lambda >0$ and a eigenmeasure $M$, such that, $L_{-\beta c}^{*} M= \lambda M$.
\end{theorem}
\begin{proof}
  (a) It is easy to see that $L_{-\beta c}$ is positive and $L_{-\beta c}$ is just the sum of the composition of continuous functions.

  (b) It is a direct application of the Tychonoff-Schauder theorem to the continuous operator $T$ given by
  $$T(\mu)=\frac{1}{\int_X L_{-\beta c}(1)(x)d\mu(x) } L_{-\beta c}^{*} (\mu).$$
  Therefore, there exists $\mu$ such that $T(\mu)=\mu$, in other words, $\mu$ satisfies
  $$\frac{1}{\int_X L_{-\beta c}(1)(x)d\mu(x) } L_{-\beta c}^{*} (\mu)=\mu.$$

Finally, take $$M=\mu \text{ and } \lambda = \int_X L_{-\beta c}(1)(x)d\mu(x) >0.$$
\end{proof}

\begin{theorem}\label{quasi inv non sep}
  Let $M$ be a measure such that $L_{-\beta c}^{*} M= \lambda M$, $\lambda>0$. Then, $M$ is quasi-invariant.
\end{theorem}
\begin{proof}

The quasi-invariant equation, when $\nu^{x}$ is the counting measure is
$$\int\sum_{t=1}^{k} h(\psi_{t}(x),\ x) dM(x)=\int\sum_{t=1}^{k} h(x,\ \psi_{t}(x))e^{-\beta c(x,\ \psi_{t}(x))} dM(x),$$
which is equivalent to
$$\int f(x) dM(x)=\int g(x) dM(x)$$
$$\int f(x) \lambda dM(x)=\int g(x) \lambda dM(x)$$
$$\int f(x) d L_{-\beta c}^{*} M(x)=\int g(x) dL_{-\beta c}^{*} M(x)$$
$$\int L_{-\beta c}(f)(x) dM(x)=\int L_{-\beta c}(g)(x) dM(x),$$
where $f(x):=\sum_{t=1}^{k} h(\psi_{t}(x),\ x)$ and $g(x):= \sum_{t=1}^{k} h(x,\ \psi_{t}(x))e^{-\beta c(x,\ \psi_{t}(x))}$.

We will evaluate $A= L_{-\beta c}(f)(x)$ and $B=L_{-\beta c}(g)(x)$:
$$A= L_{-\beta c}(f)(x)=\frac{1}{d}\sum_{j=1}^{d} \sum_{i=1}^{k} f(\psi_i(j*x))\, e^{-\beta c(j*x,\ \psi_i(j*x))}=$$
$$= \frac{1}{d}\sum_{j=1}^{d} \sum_{i=1}^{k} \sum_{t=1}^{k} h(\psi_{t}(\psi_i(j*x)),\ \psi_i(j*x)) \, e^{-\beta c(j*x,\ \psi_i(j*x))}=$$
$$= \frac{1}{d}\sum_{j=1}^{d} \sum_{i=1}^{k} \sum_{t=1}^{k} h(\psi_{t}(j*x)),\ \psi_i(j*x)) \, e^{-\beta c(j*x,\ \psi_i(j*x))},$$
because we can take $\psi_{t}(\psi_i(j*x))= \psi_{t}(j*x)$ (from Proposition~\ref{order relat struct}).

On the other hand
$$B= L_{-\beta c}(g)(x)=\frac{1}{d}\sum_{j=1}^{d} \sum_{i=1}^{k} g(\psi_i(j*x))\, e^{-\beta c(j*x,\ \psi_i(j*x))}=$$
$$=\frac{1}{d}\sum_{j=1}^{d} \sum_{i=1}^{k} \sum_{t=1}^{k} h(\psi_i(j*x),\ \psi_{t}(\psi_i(j*x)))e^{-\beta c(\psi_i(j*x),\ \psi_{t}(\psi_i(j*x)))}\, e^{-\beta c(j*x,\ \psi_i(j*x))}=$$
$$=\frac{1}{d}\sum_{j=1}^{d} \sum_{i=1}^{k} \sum_{t=1}^{k} h(\psi_i(j*x),\ \psi_{t}(\psi_i(j*x)))e^{-\beta c(j*x,\  \psi_{t}(\psi_i(j*x)))}=$$
$$=\frac{1}{d}\sum_{j=1}^{d} \sum_{i=1}^{k} \sum_{t=1}^{k} h(\psi_i(j*x),\ \psi_{t}(j*x))e^{-\beta c(j*x,\  \psi_{t}(j*x))},$$
because we can take $\psi_{t}(\psi_i(j*x))= \psi_{t}(j*x)$ (from Proposition~\ref{order relat struct}).

We claim that, for each $1 \leq j \leq d$ we have
$$\sum_{i=1}^{k} \sum_{t=1}^{k} h(\psi_{t}(j*x)),\ \psi_i(j*x)) \, e^{-\beta c(j*x,\ \psi_i(j*x))}=$$
$$\sum_{i=1}^{k} \sum_{t=1}^{k} h(\psi_i(j*x),\ \psi_{t}(j*x)) e^{-\beta c(j*x,\  \psi_{t}(j*x))},$$
so $A=B$. Indeed, for each $1 \leq  j \leq d$ we have the equality by changing the role of $t$ and $i$, and this proves our claim.
\end{proof}

\begin{corollary} \label{opr}
If $c$ is continuous, the measure $M$, given by  Theorem \ref{haar-ruelle propert}  is $(c,\beta)-$quasi-invariant.
\end{corollary}

\section{Separable cocycles} \label{sepa}
For separable cocycles we can find quasi-invariant measures by means of a more simple operator.
To do that we rewrite the quasi-invariant condition
$$\displaystyle \int\int h(s,\ x)d\nu^{x}(s)dM(x)=\int\int h(x,\ s)e^{-\beta (V(s) - V(x))}d\nu^{x}(s)dM(x)$$
on the form
$$\displaystyle \int\int g(s,\ x) e^{-\beta V(s)} d\nu^{x}(s)dM(x)=\int\int g(x,\ s)e^{-\beta V(s)}d\nu^{x}(s)dM(x),$$
by taking $h(s,x):=g(s,\ x) e^{-\beta V(s)}$.
Using the fact that the measure $\nu^{x}$ is the counting measure we obtain
$$ \int\sum_{t=1}^{k} g(\psi_t(x),\ x) e^{-\beta V(\psi_t(x))} dM(x)=\int\sum_{t=1}^{k} g(x,\ \psi_t(x))e^{-\beta V(\psi_t(x))}dM(x).$$

\medskip

By abuse of language if $M$ is quasi-invariant for $c(x,y) = V(y)- V(x)$ we may say that $M$ is quasi-invariant for $V$.

\medskip

\begin{definition}
  Let $c(x,y)=V(y)-V(x)$ be a separable cocycle, for a continuous (or, H\"older) potential $V$. We introduce the separable Haar-Ruelle operator, $L_{-\beta V}$ by choosing $\varphi(x,y)=-\beta V(y)$ that is,
  $$L_{-\beta V}(f)(x)= \frac{1}{d}\sum_{j=1}^{d} \int_{s \in [j*x]} f(s)\, e^{-\beta V(s)}d \nu^{j*x}(s),$$
  for any integrable $f:X \to \mathbb{R}$.
\end{definition}

\begin{example} \label{gen}
In the case $x\sim y \Leftrightarrow \sigma x= \sigma y$, and $\nu^{j*x}$ being the counting measure, we get that the separable Haar-Ruelle operator is
$$L_{-\beta V}(f)(x)= \frac{1}{d}\sum_{j=1}^{d} \sum_{i=1}^{d} f(i*x)\, e^{-\beta V(i*x)}= \sum_{i=1}^{d} f(i*x)\, e^{-\beta V(i*x)},$$
which is the classical Ruelle operator associated to the potential $-\beta \,V$.
\end{example}

The same arguments used in Theorem~\ref{haar-ruelle propert} proves the next theorem in the separable case.
\begin{theorem}\label{separable haar-ruelle propert} Assuming that $V$ is just continuous
  consider the separable Haar-Ruelle operator, $L_{-\beta V}$. Then,\\
  a) $L_{-\beta V}$ is positive and preserves $C^0$,\\
  b) There exists an  eigenmeasure $M$, such that, $L_{-\beta V}^{*} M= \lambda M$, for some positive value $\lambda$.
\end{theorem}

We will present an specific example which will help the reader in understanding how the above theorem can be applied for getting quasi-stationary probabilities on our setting.

\begin{example}\label{examp 1 3}
In this example, the equivalence relation $x\sim y \Leftrightarrow x_i=y_i, i\neq 1\text{ and }3$, is  continuous with respect to the symbolic structure because the correspondence
$ k*x \stackrel{\psi}{\to}  [k*x]$
is given by,
$$\psi_{ij}(k*x)=\psi_{ij}((k, x_1, x_2, x_3,...))=  (i, x_1, j, x_3,...)\,,\, 1 \leq i,j \leq d,$$
for all $1\leq k \leq d$.
Notice that $\sharp [y] =d^2$ for all $y$ and $\sigma(i, x_1, j, x_3,...)=x$ only if $j=x_2$.

\begin{figure}[!ht]  \centering  \includegraphics[width=8cm, height=4.5cm]{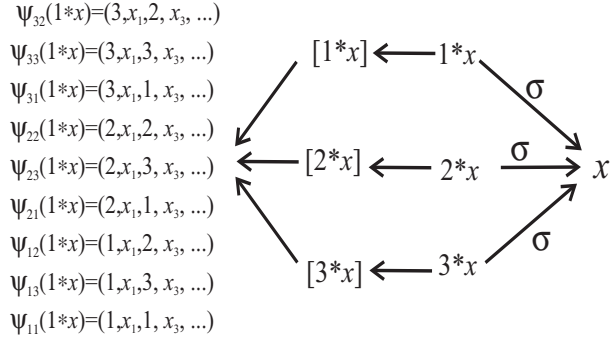}  \caption{Description of $[k*x]=\{(i, x_1, j, x_3,...)| 1 \leq i,j \leq d\}$, for $d=3$.} \label{fig rel 1 e 3} \end{figure}

In this case, $\nu^{k*x}$ being the counting measure, the Haar-Ruelle operator will be,
$$L_{-\beta V}(f)(x)= \frac{1}{d}\sum_{k=1}^{d} \int_{s \in [k*x]} f(s)\, e^{-\beta V(s)}d \nu^{k*x}(s)=$$ $$= \sum_{i,j=1}^{d} f((i, x_1, j, x_3,...))\, e^{-\beta V((i, x_1, j, x_3,...))}=$$
$$= \sum_{i,j=1}^{d} f(\tau_{ij}(x))\, e^{-\beta V(\tau_{ij}(x))}.$$

This is exactly the Hutchinson-Barnsley operator for a contractive (${\rm Lip}(\tau_{ij})= 1/8$) IFSw $(\Sigma, \tau_{ij}, q_{ij})$, where $\tau_{ij}(x)=(i, x_1, j, x_3,...)$ and $q_{ij}(x)=e^{-\beta V(\tau_{i j} \,\,x)}$.

By Theorem~\ref{Fan HB Thm}, there is a unique positive function $\varphi$ such that $L_{-\beta V}(\varphi)(x)= \lambda \varphi$ and a measure $\mu$ such that $L_{-\beta V}^*(\mu) = \lambda \mu$. We claim that $M=\mu$ is a quasi-invariant measure.

Using the fact that $\nu^{x}$ is the counting measure and
$$[x]=\{(r, x_2, t, x_4,...)| 1 \leq r,t \leq d\}$$
we must to prove that for any $g$
$$\int\sum_{r,t=1}^{d} g((r, x_2, t, x_4,...),\ x) e^{-\beta V((r, x_2, t, x_4,...))} dM(x)=$$
\begin{equation} \label{gg} \int\sum_{r,t=1}^{d} g(x,\ (r, x_2, t, x_4,...))e^{-\beta V((r, x_2, t, x_4,...))} dM(x).\end{equation}
Multiplying each side by $\lambda$ and assuming that $L_{-\beta V}^*(M) = \lambda M$ we rewrite the lhs of the above equation as
$$A:=\int\sum_{r,t=1}^{d} g((r, x_2, t, x_4,...),\ x) e^{-\beta V((r, x_2, t, x_4,...))} \lambda dM(x)= $$
$$= \int L_{-\beta V} \sum_{r,t=1}^{d} g((r, x_2, t, x_4,...),\ x) e^{-\beta V((r, x_2, t, x_4,...))}   dM(x)=$$
$$= \int \sum_{y=\tau_{ij}(x)}  \sum_{r,t=1 }^{d} g((r, y_2, t, y_4,...),\ y) e^{-\beta V((r, y_2, t, y_4,...))}  \, e^{-\beta V(y)} dM(x),$$
where $y=\tau_{ij}(x)=(i, x_1, j, x_3,...)$, so $y_2=x_1$ and  $y_4=x_3$, etc.

Thus, we get
$$A= \int \sum_{i,j=1 }^{d}  \sum_{r,t=1 }^{d} g((r, x_1, t, x_3,...),\ (i, x_1, j, x_3,...)) e^{-\beta V((r, x_1, t, x_3,...))} $$ $$ \, e^{-\beta V((i, x_1, j, x_3,...))} dM(x).$$

Rewriting the rhs of the equation (\ref{gg}) we get
$$B:= \int\sum_{r,t=1}^{d} g(x,\ (r, x_2, t, x_4,...))e^{-\beta V((r, x_2, t, x_4,...))} \lambda dM(x)=$$
$$= \int \sum_{i,j=1 }^{d}  \sum_{r,t=1 }^{d} g((i, x_1, j, x_3,...),\ (r, x_1, t, x_3,...)) e^{-\beta V((i, x_1, j, x_3,...))} $$ $$\, e^{-\beta V((r, x_1, t, x_3,...))} dM(x).$$

Thus, $A=B$, and this shows that $M$ is a quasi-invariant probability.
\end{example}

Returning to the general case and inspired by the Example~\ref{examp 1 3}  we will obtain a  fundamental result.
\begin{theorem}\label{quasi inv sep} 
  Let $M$ be a measure such that $L_{-\beta V}^{*} M= \lambda M$, $c(x,y)=V(y)-V(x)$ a separable cocycle and $V$ a continuous function. Then, $M$ is quasi-invariant for $c$.
\end{theorem}
\begin{proof} As $c$ is a separable cocycle, we have to prove that
$$ \int\sum_{t=1}^{k} h(\psi_t(x),\ x) e^{-\beta V(\psi_t(x))} dM(x)=\int\sum_{t=1}^{k} h(x,\ \psi_t(x))e^{-\beta V(\psi_t(x))}dM(x).$$

 Assume that $L_{-\beta V}^{*} M= \lambda M$ (by Theorem~\ref{separable haar-ruelle propert}).
The quasi-invariant condition is equivalent to
$$\int f(x) dM(x)=\int g(x) dM(x)$$
$$\int f(x) \lambda dM(x)=\int g(x) \lambda dM(x)$$
$$\int f(x) d L_{-\beta V}^{*} M(x)=\int g(x) dL_{-\beta V}^{*} M(x)$$
$$\int L_{-\beta V}(f)(x) dM(x)=\int L_{-\beta V}(g)(x) dM(x),$$
where
$$f(x):=\sum_{t=1}^{k} h(\psi_t(x),\ x) e^{-\beta V(\psi_t(x))}$$ and
$$g(x)= \sum_{t=1}^{k} h(x,\ \psi_t(x))e^{-\beta V(\psi_t(x))}.$$

We will evaluate $A= L_{-\beta V}(f)(x)$ and $B=L_{-\beta V}(g)(x)$:
$$A= L_{-\beta V}(f)(x)=\frac{1}{d}\sum_{j=1}^{d} \sum_{i=1}^{k} f(\psi_i(j*x))\, e^{-\beta V(\psi_i(j*x))}=$$
$$= \frac{1}{d}\sum_{j=1}^{d} \sum_{i=1}^{k} \sum_{t=1}^{k} h(\psi_{t}(\psi_i(j*x)),\ \psi_i(j*x))e^{-\beta V(\psi_{t}(\psi_i(j*x)))} \, e^{-\beta V( \psi_i(j*x))}=$$
$$= \frac{1}{d}\sum_{j=1}^{d} \sum_{i=1}^{k} \sum_{t=1}^{k} h(\psi_{t}(j*x)),\ \psi_i(j*x)) \, e^{-\beta (V(\psi_t(j*x)) + V( \psi_i(j*x)))},$$
because we can take $\psi_{t}(\psi_i(j*x))= \psi_{t}(j*x)$\,\, (from Proposition~\ref{order relat struct}).

On the other hand, by an analogous computation
$$B= L_{-\beta V}(g)(x)=\frac{1}{d}\sum_{j=1}^{d} \sum_{i=1}^{k} g(\psi_i(j*x))\, e^{-\beta V(\psi_i(j*x))}=$$
$$= \frac{1}{d}\sum_{j=1}^{d} \sum_{i=1}^{k} \sum_{t=1}^{k} h(\psi_{i}(j*x)),\ \psi_t(j*x)) \, e^{-\beta (V( \psi_i(j*x))) +V(\psi_t(j*x))},$$
because we can take $\psi_{t}(\psi_i(j*x))= \psi_{t}(j*x)$ \,(from Proposition~\ref{order relat struct}).

Obviously $A=B$ which proves our claim.
\end{proof}

\begin{remark}
  We observe that in Corollary \ref{coro}, the measure
$M_0$ such that $L_{-\beta V}^{*} M_0= \lambda_0 M_0$, given by Theorem~\ref{separable haar-ruelle propert}, is quasi-invariant.
In other words, we proved that
$$ \int\sum_{t=1}^{k} h(\psi_t(x),\ x) e^{-\beta V(\psi_t(x))} dM_0(x)=\int\sum_{t=1}^{k} h(x,\ \psi_t(x))e^{-\beta V(\psi_t(x))}dM_0(x).$$

However, $L_{-\beta c}^{*} M_1= \lambda_1 M_1$ by Theorem~\ref{haar-ruelle propert} and, moreover,  by Corollary \ref{opr} $M_1$ is quasi-invariant, that is
$$ \int\sum_{t=1}^{k} h(\psi_t(x),\ x)  dM_1(x)=\int\sum_{t=1}^{k} h(x,\ \psi_t(x))e^{-\beta ( V(\psi_t(x))-V(x))} dM_1(x),$$
because $c(x,y)=V(y)-V(x)$.

Thus, $M_0$ and $M_1$ are quasi-invariant measures for the same separable cocycle $c(x,y)=V(y)-V(x)$ and they are not necessarily equal (they are eigenmeasures of different operators). This abundance of quasi-invariant measures will be explored in the next section. Note that for the more particular groupoid considered on  the setting of \cite{CELM} it is also shown that the quasi-invariant probability is not unique (there the cocycle was Holder). The groupoid $G$ of the mentioned result on \cite{CELM} is the one presented in Example \ref{gen} (a continuous groupoid).
\end{remark}

\begin{corollary}\label{coro}
If $V$ is in continuous, the measure $M$, such that, $L_{-\beta V}^{*} M= \lambda M$, is quasi-invariant for the associated $c$.
 \end{corollary}

For the proof see Theorem \ref{quasi inv sep} (note that $M$ exists by Theorem \ref{separable haar-ruelle propert}).

\section{Quasi-invariant measures arising from IFS}

On this section we will assume some more regularity on the cocycle.

The terminology IFSw (IFS with weights) was introduced in \cite{LO} and \cite{CO} for the case of a IFS where the weights are not normalized. This case was also considered  in \cite{FL} without giving a special name.

We recall a fundamental result on the Hutchinson-Barnsley operator for an IFSw (see \cite{FL}).

\begin{theorem}\label{Fan HB Thm}
Suppose $0\leq j \leq n-1$ and  $R=(Y, \tau_j, q_j)$ an IFSw satisfying the following hypothesis \\
a) $\tau_j$ is a contraction and \\
b) each $q_j$ is Dini continuous that is, $\displaystyle\int_{0}^{1}\max_{d(x,y)\leq t}\frac{|\ln q_j(x) - \ln q_j(y)|}{t}dt<+\infty$\\
and $\rho >0$, the spectral radius of $\displaystyle B_{R}(f)(x)=\sum_{j=0}^{n-1} q_j(x) f(\tau_j x)$, restricted to the attractor
$K$ of $\mathcal{R}=(Y, \tau_j)$, where
$$\bigcup_{j=0}^{n-1} \tau_j (K)=K,$$
then, there exists  a unique $h >0$ and a unique $\mu  \in M(K)$, such that, $$B_{R}(h )=\rho  \,h  \; \mathcal{L}_{R}(\mu )= \rho\,  \mu \text{ and } \int h \,d\mu=1,$$
where $\mathcal{L}_{R}=B_{R}^*$, and $\alpha =h \, \mu $ is a probability called the Gibbs measure of the system. Moreover, for every $f_0 \in C(K)$ we get $\rho^{-n}\, B_{R}^{n}(f_0) \to (\int f_0 d\mu ) h $, and for any $\mu_0 \in M(K)$  we get $\rho^{-n} \mathcal{L}_{R}^{n}(\mu_0) \to (\int h  d\mu_0) \mu $.
\end{theorem}

\medskip

The proof of the next corollary is analogous to the one in the section Boundary conditions on \cite{CL}.

\medskip
\begin{corollary}\label{approx eigenmeasure fan}
   Under the same hypothesis of the Theorem~\ref{Fan HB Thm} we assume that $\mu$ satisfies $\mathcal{L}_{R}(\mu )= \rho\,  \mu$. Then, we get for any $x_0 \in X$
  $$\lim_{n\to\infty}\frac{B_{R}^{n}(f)(x_0)}{B_{R}^{n}(1)(x_0)}= \int f(x) d\mu(x),$$
  for any $f \in C^0$.
\end{corollary}
\begin{proof}
  First we choose a  point $x_0$. By Theorem~\ref{Fan HB Thm}, for every $f_0 \in C(K)$ we get $\rho^{-n}\, B_{R}^{n}(f_0) \to (\int f d\mu ) h $, and for any $\mu_0 \in M(K)$  we get that $\rho^{-n} \mathcal{L}_{R}^{n}(\mu_0) \to (\int h \, d\mu_0) \mu $.\\
  If $f_0=1$, then $\rho^{-n}\, B_{R}^{n}(1)(x_0) \to (\int 1\, d\mu ) h(x_0)= h(x_0)$. Now, if $\mu_0=\delta_{x_0}$,  then, $\rho^{-n} \mathcal{L}_{R}^{n}(\delta_{x_0}) \to (\int h  \,d\delta_{x_0}) \mu $, or equivalently, for any $f \in C^0$ we have
  $\rho^{-n}  B_{R}^{n}(f)(x_0) \to h(x_0)\,\int f(x) \,d \mu(x)$.

  From this, we can compute the limit
  $$\lim_{n\to\infty}\frac{B_{R}^{n}(f)(x_0)}{B_{R}^{n}(1)(x_0)}=
  \lim_{n\to\infty}\frac{\rho^{-n} B_{R}^{n}(f)(x_0)}{\rho^{-n} B_{R}^{n}(1)(x_0)}=$$
  $$= \frac{h(x_0)\int f(x) d \mu(x)}{h(x_0)}=\int f(x)d  \mu(x).$$
\end{proof}

\medskip
\begin{remark} \label{kr}
The above result shows that we can approximate the eigenprobability $\mu$ by means of the backward iteration of the dynamics of $\sigma$. This method   is analogous to the use of the thermodynamic limit with boundary conditions in order to get Gibbs probabilities in  Statistical Mechanics (see \cite{CL}).
\end{remark}

\medskip

There is an important class of examples leading us to consider IFS.

\begin{example}
In this example, the equivalence relation is $x\sim y \Leftrightarrow x_i=y_i, i \not\in \{n_1, ..., n_r\}$, with $n_1=1$,
where we fixed the set $\{n_1, ..., n_r\} \subset \mathbb{N}.$  This equivalence relation is  continuous with respect to the symbolic structure because the correspondence
$ k*x \stackrel{\psi}{\to}  [k*x]$
is given by
$$\psi_{i_1...i_r}(k*x)=\{y \ | \ y_{i}=(k*x)_i \,,\text{ for }  i \not\in \{n_1, ..., n_r\} \},$$
which is constant with respect to $k$.  Notice that $\sharp [y] =d^r$ for all $y$ and,  in general,  $\sigma(y)\neq x$.

In this case, $\nu^{k*x}$ being the counting measure, the separable Haar-Ruelle operator will be,
$$L_{-\beta V}(f)(x)= \frac{1}{d}\sum_{k=1}^{d} \int_{s \in [k*x]} f(s)\, e^{-\beta V(s)}d \nu^{k*x}(s)=$$ $$=\frac{1}{d}\sum_{k=1}^{d}  \sum_{i_1...i_r=1}^{d} f(\psi_{i_1...i_r}(k*x))\, e^{-\beta V(\psi_{i_1...i_r}(k*x))}= $$
$$\sum_{i_1...i_r=1}^{d} f(\psi_{i_1...i_r}(1*x))\, e^{-\beta V(\psi_{i_1...i_r}(1*x))}.$$
That is, exactly the Hutchinson-Barnsley operator for a contractive IFSw
$$(\Sigma, \psi_{i_1...i_r}(1*x), e^{-\beta V(\psi_{i_1...i_r}(1*x))}).$$
\end{example}

\begin{remark}\label{how to approxim quasi IFS}
The application of the Theorem~\ref{Fan HB Thm} to the IFSw $R=(Y, \tau_j, q_j)$ is immediate, when $\tau_j$ is a contraction and $q_j(x)=e^{-\beta V(\tau_j x)}$, for a H\"older (or, Lipschitz) potential $V$. In this case,  the Haar-Ruelle operator is the Hutchinson-Barnsley operator for a contractive IFSw. Then, the eigenprobability is a quasi-invariant probability for  $V$. It will follow from Theorem \ref{quasi inv sep}  and Theorem  \ref{Fan HB Thm} that the  quasi-invariant probability $M=\mu$ will have support on the attractor $K$. Note that $K$ can be eventually smaller than $X$.

Moreover, by Remark \ref{kr} one can get a computational way to  approximate the integral  $\int f\, d M$.

The bottom line is: the dynamics helps on finding approximations of the quasi-invariant probability  $M$ when $V$ is H\"older.
\end{remark}

Given the separable Haar-Ruelle operator
$$L_{-\beta V}(f)(x)=\frac{1}{d}\sum_{j=1}^{d} \sum_{i=1}^{k} f(\psi_i(j*x))\, e^{-\beta V(\psi_i(j*x))},$$
we consider the  associated IFSw
$$R=(\Sigma, \psi_{i}(j*x), e^{-\beta V(\psi_{i}(j*x))}),$$
where $1\leq j\leq d$ and $1\leq i\leq k$.
Then,
$$d L_{-\beta V}(f)(x)=B_{R}(f)=\sum_{ij} f(\psi_i(j*x))\, e^{-\beta V(\psi_i(j*x))}.$$

\begin{theorem}\label{reduct separable to IFS}
 Suppose that $c(x,y)=V(y)-V(x)$  is a separable cocycle. If the representation $\psi$ of $\sim$ is Lipschitz, with $\max_{ij}Lip(\psi_{i}(j* x)) < 1$, and $V(x)$ is $\alpha$-H\"older, then $M=\mu$ is quasi-invariant, where $\mu$ is the eigenmeasure of $B_{R}^*$ given by the Theorem~\ref{Fan HB Thm}.
\end{theorem}
\begin{proof}
   We consider the associated IFSw $R=(\Sigma, \psi_{i}(j*x), e^{-\beta V(\psi_{i}(j*x))})$ then the IFS $R=(\Sigma, \psi_{i}(j*x))$ is contractive: $$d( \psi_{i}(j*x),\  \psi_{i}(j*z))  \leq \max_{ij} Lip(\psi_{i}(j*\cdot)) d(x,z).$$
   Note that the weights are Dini continuous:
   $$\int_{0}^{1}\max_{d(x,z)\leq t}\frac{|\ln e^{-\beta V(\psi_{i}(j*x))} - \ln e^{-\beta V(\psi_{i}(j*z))}|}{t}dt=$$
   $$=\,\beta\int_{0}^{1}\max_{d(x,z)\leq t}\frac{| V(\psi_{i}(j*x)) - V(\psi_{i}(j*z))|}{t}dt \leq $$
   $$\leq \beta\int_{0}^{1}\max_{d(x,z)\leq t} \frac{d(\psi_{i}(j*x) ,\ \psi_{i}(j*z))^{\alpha}}{t}dt \leq $$
   $$\leq\beta \max_{ij}Lip(\psi_{i}(j* x))^{\alpha}\int_{0}^{1} t^{\alpha -1}dt \leq$$
   $$\leq\beta \max_{ij}Lip(\psi_{i}(j* x))^{\alpha}  \frac{1}{\alpha}<+\infty.$$
   By Theorem~\ref{Fan HB Thm} there exists a probability $\mu$, such that, $B_{R}^*(\mu)=\rho \mu$.  Since $L_{-\beta V} =\frac{1}{d} B_{R}$ we get
   $$L_{-\beta V}^*(\mu)= \frac{1}{d} B_{R}^* (\mu)=  \frac{1}{d} \rho \mu= \lambda \mu,$$
   where $\lambda:=\frac{\rho}{d}$. By Theorem~\ref{quasi inv sep}, $M=\mu$ is quasi-invariant.
\end{proof}

\begin{example}\label{examp 3}
In this example $X=\{1,2\}^{\mathbb{N}}$ and we will make explicit computations in a case which it is not an IFSw given by the two inverse branches of the shift map. Let $V(x)=V(x_1)=\frac{1}{4}(x_1 -1)^2$ be a  potential depending only on the first coordinate.

We consider the equivalence relation $x\sim y \Leftrightarrow x_k=y_k, k\neq 3$, which is obviously continuous with respect to the symbolic structure because the correspondence
$ j*x \stackrel{\psi}{\to}  [j*x]$
is given by,
$$\psi_{i}(j*x)=\psi_{i}((j, x_1, x_2, x_3,...))=(j, x_1, i, x_3,...),\, 1 \leq i \leq 2,$$
for all $1\leq j \leq 2$.
Notice that $\sharp [y] =2$ for all $y\in X$ and $\sigma(j, x_1, i, x_3,...)=x$ only if $i=x_2$.
The separable Haar-Ruelle operator is
$\displaystyle L_{-\beta V}(f)(x)=\frac{1}{2}\sum_{i,j=1}^{2} f(\psi_i(j*x))\, e^{-\beta V(\psi_i(j*x))}.$
We consider the associated IFSw
$R=(X, \psi_{i}(j*x), e^{-\beta V(\psi_{i}(j*x))}),$
where, $1\leq j \leq 2$ and $1\leq i\leq 2$.
The IFS $R=(X, \psi_{i}(j*x))$ is given by $4$ maps  $\psi_{i}(j*x)$ for $1\leq i, \ j \leq 2$(and, not two) is contractive.
Indeed, notice that
$$d( \psi_{i}(j*x) , \psi_{i}(j*z))
=d((j, x_1, i, x_3,...),\ (j, z_1, i, z_3,...))\leq $$
$$\leq \frac{1}{2} d((x_1, i, x_3,...), \ (z_1, i, z_3,...))
\left\{
  \begin{array}{ll}
    =\frac{1}{2} d(x, \ z), &  x_1\neq z_1 \\
    < \frac{1}{4} d(x, \ z), & x_k=z_k, \ k\leq 2 \\
    =\frac{1}{2} d(x, \ z), & x_k=z_k, \ k\leq 3
  \end{array}
\right.
$$
thus ${\rm Lip}(\psi_{i}(j*x))= \frac{1}{2}$. In particular, $\max_{ij}Lip(\psi_{i}(j* x)) < 1$.  From the definition $V(x)=V(x_1)=\frac{1}{4}(x_1 -1)^2$ then,
$$|V(x) -V(z)|=
\left\{
  \begin{array}{ll}
    0, & x_1=z_1 \\
    \frac{1}{4}, & \text{otherwise}
  \end{array}
\right.
$$
If $ x_1 \neq z_1$ then $d(x,z)=\frac{1}{2}$ and $|V(x) -V(z)|=\frac{1}{4}=d(x,z)^2$. If $x_1=z_1$ then $|V(x) -V(z)|=0 \leq d(x,z)^2$. We conclude that $V(x)$ is $\alpha$-H\"older, for $\alpha=2$. From Theorem~\ref{reduct separable to IFS}, the probability $M=\mu$ is quasi-invariant - taking $\mu$  the eigenmeasure of $B_{R}^*$ given by the Theorem~\ref{Fan HB Thm}  - where $B_{R}^*$ is the dual of
$$B_{R}(f)(x)=2L_{-\beta V}(f)(x)=\sum_{i,j=1}^{2} f(\psi_i(j*x))\, e^{-\beta V(\psi_i(j*x))}.$$

Following  Corollary~\ref{approx eigenmeasure fan} we get that $\mu$ satisfies,  for any $x^0 \in X$,
$$\lim_{n\to\infty}\frac{B_{R}^{n}(f)(x^0)}{B_{R}^{n}(1)(x^0)}= \int f(x) d\mu(x),$$
for any $f \in C^0$.

To make those computations we need to understand how the orbits and the iterates of the operator $B_R$ behaves.  Given $x=(x_1, x_2, x_3, ...) \in X$ we describe its orbit by $x^0=x$, $x^{1}=\psi_{i_{0}}(j_{0}*x^{0})$,  $x^{2}=\psi_{i_{1}}(j_{1}*x^{0})$, ....

More explicitly, we have
$$x^{0}=(x_1,  x_2, x_3, ...), \; x^{1}=(j_{0},x_1,  i_{0}, x_3, ...),$$
$$x^{2}=(j_{1},j_{0}, i_{1},  i_{0}, x_3, ...), \; x^{3}=(j_{2},j_{1}, i_{2},  i_{1},  i_{0}, x_3, ...), \;\text{ etc.}$$
Therefore, for each $n$ we get $x^{n}=(j_{n-1},j_{n-2},i_{n-1}, ...,  i_{2},  i_{1},  i_{0}, x_3, ...).$

From this, we can write
$$B_{R}(f)(x)=\sum_{i_1,j_1=1}^{2} f(\psi_{i_{1}}(j_{1}*x))\, e^{-\beta V(\psi_{i_{1}}(j_{1}*x))}$$
$$B_{R}^2(f)(x)=\sum_{i_0,j_0=1}^{2} B_{R}(f)(\psi_{i_{0}}(j_{0}*x))\, e^{-\beta V(\psi_{i_{0}}(j_{0}*x))}=$$
$$= \sum_{i_0,i_1,j_0,j_1=1}^{2}   f(\psi_{i_{1}}(j_{1}*(\psi_{i_{0}}(j_{0}*x))))\, e^{-\beta V(\psi_{i_{1}}(j_{1}*(\psi_{i_{0}}(j_{0}*x))))}  \, e^{-\beta V(\psi_{i_{0}}(j_{0}*x))}=$$
$$= \sum_{i_0,i_1,j_0,j_1=1}^{2}   f(x^2)\, e^{-\beta (V(x^2) +  V(x^1))}= \sum_{i_0,i_1,j_0,j_1=1}^{2}   f(x^2) \, e^{-\beta (V(j_{1}) +  V(j_{0}))},$$
and the $n$-th power will be
$$B_{R}^n(f)(x)= \sum_{i_{0}, ...,i_{n-1}, j_{0}, ...,j_{n-1}=1}^{2}   f(x^n) \, e^{-\beta (V(j_{n-1}) + \cdots + V(j_{0}))}.$$

In order to make a histogram of $\mu$ we fix a length  $k\geq 2$ to built a partition $X=\bigcup \overline{a_1,..., a_k}$, and compute
$$ u(t)= \mu(\overline{a_1,..., a_k})=\int \chi_{_{\overline{a_1,..., a_k}}}(x) d\mu(x)\simeq \frac{B_{R}^{n}(f)(x^0)}{B_{R}^{n}(1)(x^0)}.$$

We plot this value at the point $t= 2^{-2}a_1+ \cdots + 2^{-k-1}a_k \in [0,1]$. In this way, each point in the histogram correspond to the measure of the associated element of the partition. We  choose $x^0=(1,1,1,1,1, ....)$ for simplicity.
\begin{figure}[!ht]  \centering
\includegraphics[width=4.2cm]{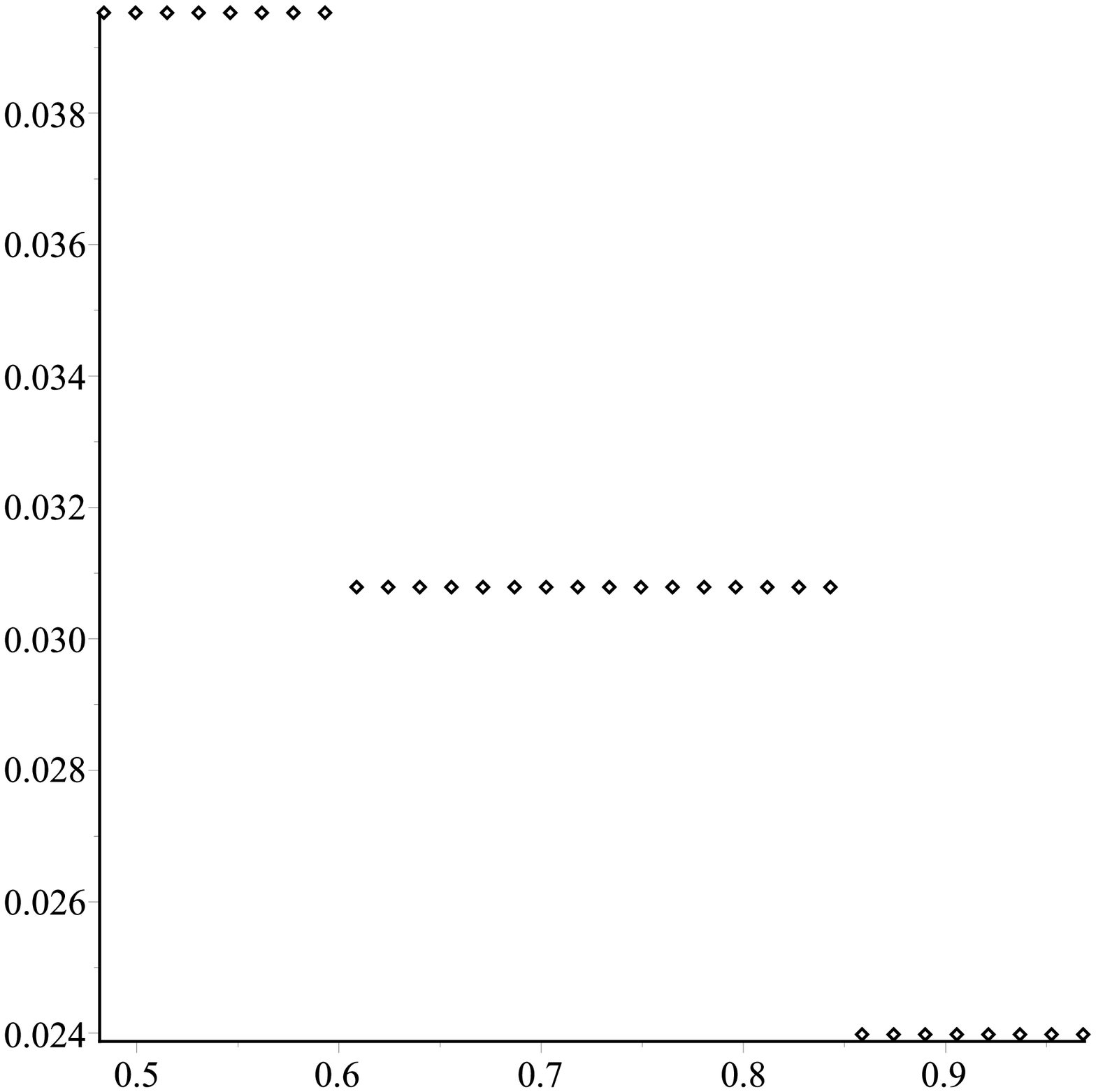}
\includegraphics[width=4.2cm]{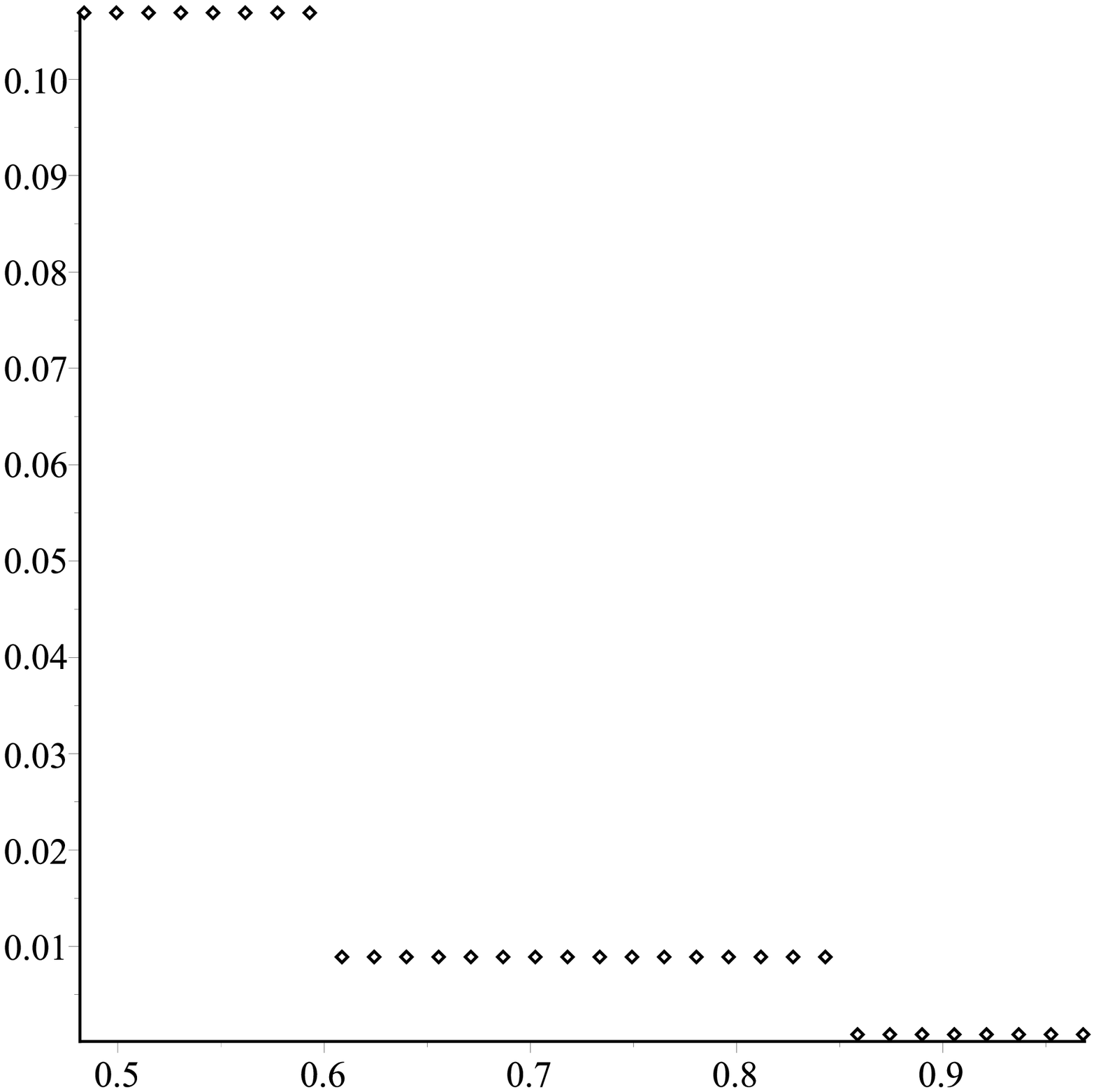}
\includegraphics[width=4.2cm]{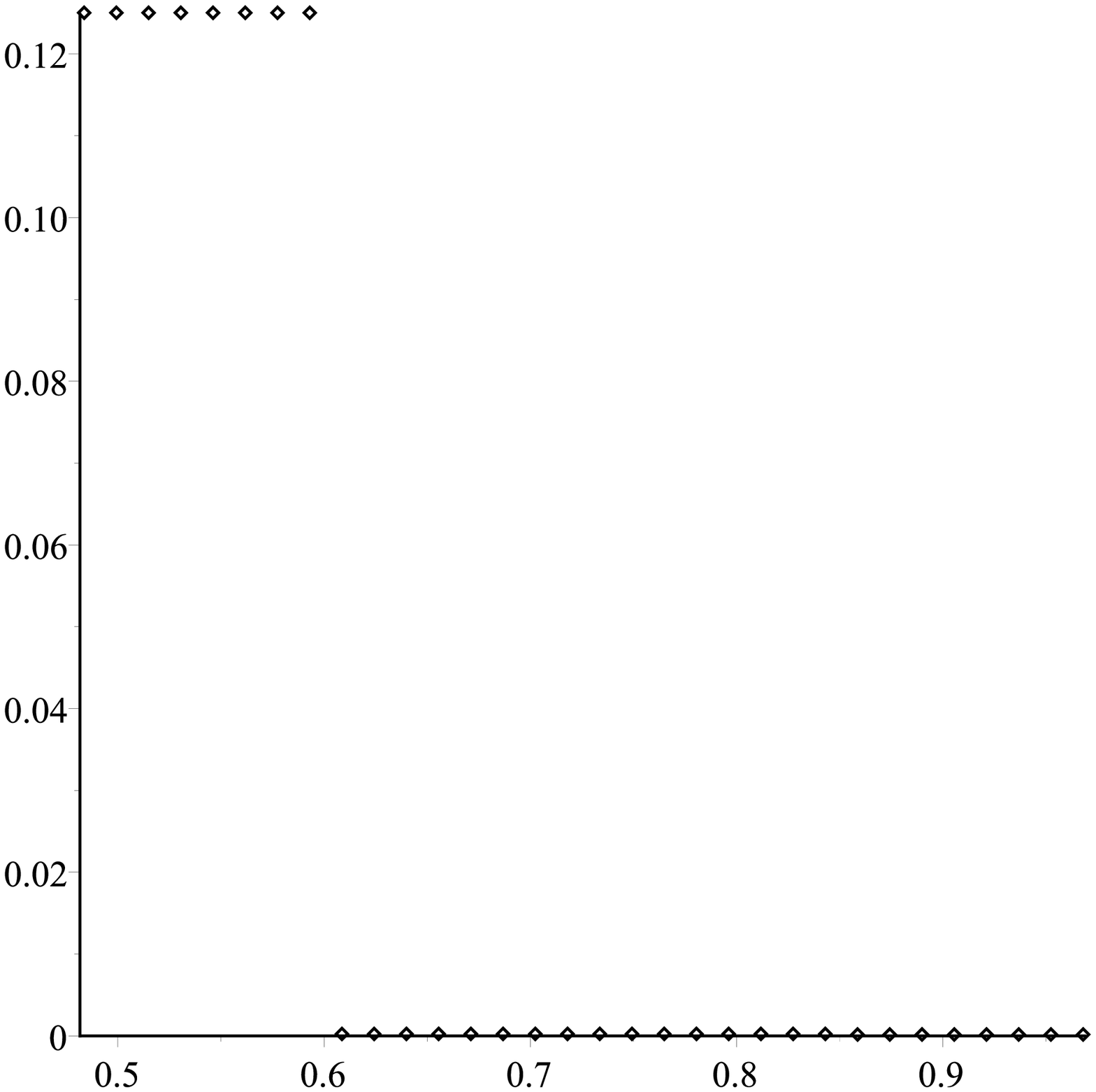}
\caption{Computation of $v(t)$ for  $k=5$ and $n=9$. In this picture we consider $\beta=1$(left), $\beta=10$(middle) and $\beta=30$(right)} \label{fig rel 3 histogram var beta} \end{figure}

In the Figure~\ref{fig rel 3 histogram var beta} we can see an approximation of the measure $\mu$ which is the only eigenmeasure of $B_{R}^{*}$ for three different values of $\beta$:  $\beta=1$, $\beta=10$ and $\beta=30$. For bigger values of $\beta$ the measure is concentrated close to the smaller values of $t$. By our representation the value $t= 2^{-2}a_1+ \cdots +2^{-6}a_5 $ corresponds to the cylinders $(\overline{1,1,1,1,1}$), $(\overline{1,1,1,1,2})$,...,$(\overline{1,1,2,2,1})$ and  $(\overline{1,1,2,2,2})$.
\end{example}

\section{The Haar Operator}
We already proved that certain eigenmeasures of the Haar-Ruelle operators are quasi-invariant measures. In this section we are going to consider necessary conditions on the quasi-invariant measure. Our goal is to show that any quasi-invariant measure is an eigenmeasure of some Haar operator. We notice that this operator is quite different from the Haar-Ruelle operator.

\begin{definition}\label{haar op}
We introduce the Haar operator, $H_{-\beta c}$ defined by
  $$H_{-\beta c}(f)(x)= \frac{1}{\nu^{x}([x])}\int_{s \in [x]} f(s)\, e^{-\beta c(x,s)}d \nu^{x}(s),$$
acting on  any integrable $f:X \to \mathbb{R}$.
\end{definition}

\begin{proposition}\label{haar op properties}
  We claim that\\
  a) $H_{-\beta c} : C^0 \to C^0$.\\
  b) $H_{-\beta c}^2= H_{-\beta c}$, in particular $V=H_{-\beta c}(f)$ for a given $f:X \to \mathbb{R}$ is a fixed point, that is, $H_{-\beta c}(V)=V$.\\
  c) $H_{-\beta c}$ is positive.\\
  d) $H_{-\beta c + b}(1)=1$, for $b(x,y)=V(y)-V(x)$, where $V=\ln H_{-\beta c}(f)$, for some $f:X \to \mathbb{R}^{+}$.
\end{proposition}
\begin{proof}
  a) $H_{-\beta c} : C^0 \to C^0$ because the map
$$x \to [x]:=\{s \in X | s \sim x\}$$
is continuous as a set function.\\\\
b) Let $V=H_{-\beta c}(f)$ for some $f:X \to \mathbb{R}$,  then,
$$ H_{-\beta c}(V)(x)= \frac{1}{\nu^{x}([x])}\int_{s \in [x]} V(s)\, e^{-\beta c(x,s)}d \nu^{x}(s)=$$
$$= \frac{1}{\nu^{x}([x])}\int_{s \in [x]} \frac{1}{\nu^{s}([s])}\int_{t \in [s]} f(t)\, e^{-\beta c(s,t)}d \nu^{s}(t)\, e^{-\beta c(x,s)}d \nu^{x}(s)=$$
$$=\frac{1}{\nu^{x}([x])}\int_{t \in [x]} f(t)\, e^{-\beta c(x,t)}d \nu^{x}(t)= V(x),$$
because $c(x,  t)=c(x,  s)+c(s,  t), \forall x\sim s\sim t$.\\\\
c) It is obvious.\\\\
d) Consider $V=\ln H_{-\beta c}(f)$, for any $f:X \to \mathbb{R}^{+}$, and $b(x,y)=V(y)-V(x)$, then,
$$H_{-\beta c + b}(1)= \frac{1}{\nu^{x}([x])}\int_{s \in [x]} 1 \, e^{-\beta c(x,s)+ b(x,s)}d \nu^{x}(s)= $$
$$= \frac{1}{\nu^{x}([x])}\int_{s \in [x]} 1 \, e^{-\beta c(x,s)+ V(s)- V(x)}d \nu^{x}(s)= $$ $$=\frac{e^{-V(x)}}{\nu^{x}([x])}\int_{s \in [x]} e^{V(s)} \, e^{-\beta c(x,s)}d \nu^{x}(s)= \frac{H_{-\beta c}(e^{V})(x)}{e^{V(x)}}=1,$$
because $H_{-\beta c}(e^{V})= e^{V}$.
\end{proof}

From the previous result we can establish a standard normalization. Since $H_{-\beta c}(1)(x)>0$ we have that $V(x)=\ln H_{-\beta c}(1)(x)$ and
$$H_{-\beta c + b}(f)(x)=\frac{1}{\nu^{x}([x])}\int_{s \in [x]} f(s) \, e^{-\beta c(x,s)+ \ln H_{-\beta c}(1)(s)- \ln H_{-\beta c}(1)(x)}d \nu^{x}(s)$$
satisfies $H_{-\beta c + b}(1)=1$, where
$$V(x)=\ln H_{-\beta c}(1)(x)= \frac{1}{\nu^{x}([x])}\int_{s \in [x]}  e^{-\beta c(x,s)}d \nu^{x}(s).$$

This kind of normalization is analogous, in some sense, to the one presented in \cite{PP}  or in \cite{LMMS}.

\section{Characterizing quasi-invariant probabilities}

Suppose that $M$ is a quasi-invariant probability for $\beta\in \mathbb{R}$ and $c:G\rightarrow \mathbb{R}$, that is, for  any $h$
$$\int\int h(s,\ x)d\nu^{x}(s)dM(x)=\int\int h(x,\ s)e^{-\beta c(x,s)}d\nu^{x}(s)dM(x).$$

Since $M$ is a probability in $G^0$, it will be completely determined by its action on $C^0$.

\begin{proposition} \label{pro}
   If $M$ is a quasi-invariant probability for $\beta\in \mathbb{R}$ and $c:G\rightarrow \mathbb{R}$, then
   $$\int_{G^0} f(x) dM(x)=\int_{G^0} H_{-\beta c}(f) (x)dM(x),$$
   for all $f \in C^0$.
\end{proposition}
\begin{proof}
  Consider $f \in C^0$ and define the integrable function
  $$h(x,y)=\frac{f(s)}{\nu^{x}([x])},$$ then,
  $$\int_{G^0}\int_{[x]}  \frac{f(x)}{\nu^{s}([s])} d \nu^{x}(s)dM(x)=\int_{G^0}\int_{[x]} \frac{f(s)}{\nu^{x}([x])} e^{-\beta c(x,s)} d \nu^{x}(s)dM(x)$$
  $$\int_{G^0} f(x) dM(x)=\int_{G^0}\frac{1}{\nu^{x}([x])}\int_{[x]} f(s) e^{-\beta c(x,s)} d \nu^{x}(s)dM(x)$$
\end{proof}

\begin{proposition}
  Consider $V(x)=\ln H_{-\beta c}(1)(x)$ and the normalization
$$H_{-\beta c + b}(f)(x)=\frac{1}{\nu^{x}([x])}\int_{s \in [x]} f(s) \, e^{-\beta c(x,s)+ V(s)- V(x)} d \nu^{x}(s),$$
where
$\displaystyle V(x)=\ln H_{-\beta c}(1)(x)= \frac{1}{\nu^{x}([x])}\int_{s \in [x]}  e^{-\beta c(x,s)}d \nu^{x}(s).$ Then, $e^{V}dM=dM^*$ is an eigenmeasure of $H_{-\beta c + b}^{*}$. Reciprocally, if $H_{-\beta c + b}^{*}(M^*)=M^*$, then $dM=e^{-V}dM^*$ satisfies $$\int_{G^0} f(x) dM(x)=\int_{G^0} H_{-\beta c}(f) (x)dM(x),$$
   for all $f \in C^0$.

\end{proposition}
\begin{proof}
  Given $e^{V}dM=dM^*$ and $f \in C^0$, we define $g = f  e^V$, then
  $$\int_{G^0} g(x)  dM(x)= \int_{G^0} H_{-\beta c}(g)(x)  dM(x)$$
  $$\int_{G^0} f(x)e^{V(x)}  dM(x)= \int_{G^0} H_{-\beta c}(f  e^V)(x)  dM(x)$$
  $$\int_{G^0} f(x)   dM^*(x)= \int_{G^0} H_{-\beta c}(f  e^V)(x)e^{-V(x)} e^{V(x)} dM(x)$$
  $$\int_{G^0} f(x)   dM^*(x)= \int_{G^0} \frac{1}{\nu^{x}([x])}\int_{[x]}f(s) e^{V(s)} e^{-\beta c(x,s)} \nu^{x}(ds) e^{-V(x)}  dM^*(x)$$
  $$\int_{G^0} f(x)   dM^*(x)= \int_{G^0} H_{-\beta c + b}(f)(x)   dM^*(x).$$
  The reciprocal is true because we can reverse the previous argument.
\end{proof}

\medskip

A. O. Lopes

arturoscar.lopes@gmail.com

partially supported by CNPq

\medskip

E. R. Oliveira

elismar.oliveira@ufrgs.br
\medskip

Instituto de Matematica  e Estatistica - UFRGS


\medskip

\begin{thebibliography}{1}

\bibitem{AB}
Charalambos~D. Aliprantis and Kim~C. Border.
\newblock {\em Infinite dimensional analysis}.
\newblock Springer, Berlin, third edition, 2006.
\newblock A hitchhiker's guide.


\bibitem{BCLMS}
A.~T. Baraviera, L. M. Cioletti, A. O. Lopes, Joana Mohr and Rafael R.
  Souza,
\newblock On the general one-dimensional {$XY$} model: positive and zero
  temperature, selection and non-selection.
\newblock {\em Rev. Math. Phys.}, 23(10):1063--1113, 2011.


\bibitem{CELM} G. Castro, A. O. Lopes and G. Mantovani, Haar systems, KMS
states on von Neumann algebras and $C^*$-algebras on dynamically defined
groupoids and Noncommutative Integration, preprint 2017


\bibitem{CL} L. Cioletti and A. O.  Lopes,
Interactions, Specifications, DLR probabilities and the
Ruelle Operator in the One-Dimensional Lattice, Discrete and Cont. Dyn. Syst. - Series A, Vol 37, Number 12, 6139 -- 6152 (2017)

\bibitem{CL1} L. Cioletti and A. O.  Lopes, Spectral Triples  on Thermodynamic Formalism and Dixmier trace representations of Gibbs measures, preprint arXiv (2018)


\bibitem{CO} L. Cioletti and E. R. Oliveira,
Thermodynamic Formalism for Iterated Function Systems with Weights, preprint Arxiv (2017)


\bibitem{Con} A. Connes,
Sur la theorie non commutative de l'integration, preprint



\bibitem{EL}
R. Exel and A. Lopes, $C^*$-Algebras, Approximately Proper Equivalence Relations, and Thermodynamic
Formalism, Ergodic Theo and Dyn. Syst., Vol 24, pp 1051-1082, Erg Theo and Dyn Syst (2004).

\bibitem{EL1} R. Exel and A. Lopes,
$C^*$- Algebras and Thermodynamic Formalism,
Sao Paulo Journal of Mathematical Sciences 2, 1 (2008), 285–-307


\bibitem{FL}
A. H. Fan and Ka-Sing Lau,
\newblock Iterated function system and {R}uelle operator.
\newblock {\em J. Math. Anal. Appl.}, 231(2):319--344, 1999.



 \bibitem{KumRen} A. Kumjian and J. Renault, KMS states on   $ C^*$-Algebras associated to expansive maps, Proc. AMS Vol. 134, No. 7, 2067-2078 (2006)


\bibitem{LM1} A. O. Lopes and G. Mantovani,
The KMS Condition for the homoclinic equivalence relation and Gibbs probabilities, preprint Arxiv 2017

\bibitem{LO}
A. O. Lopes and E. R. Oliveira,
\newblock Entropy and variational principles for holonomic probabilities of
  {IFS}.
\newblock {\em Discrete Contin. Dyn. Syst.}, 23(3):937--955, 2009.



\bibitem{LMMS} A. O. Lopes, J. K. Mengue, J. Mohr and  R. R. Souza,
Entropy and Variational Principle for  one-dimensional Lattice Systems with a general a-priori probability: positive and zero temperature,  Erg. Theory and Dyn Systems, 35 (6), 1925–-1961 2015

\bibitem{MO}
J. Mengue and E. R. Oliveira,  Duality results for iterated function systems with a general family of branches, Stoch. Dyn. 17 (2017), no. 3, 1750021, 23 pp

\bibitem{PP}
W. Parry and M. Pollicott.
\newblock Zeta functions and the periodic orbit structure of hyperbolic
  dynamics.
\newblock {\em Ast\'erisque}, (187-188):268, 1990.



\bibitem{Ren2}
J. Renault, $C^*$-Algebras and Dynamical Systems, XXVII Coloquio Bras. de Matematica - IMPA (2009)



\end{thebibliography}
\end{document}